\documentclass[11pt]{amsart}
\usepackage{amssymb,amsmath,amsfonts}

\headsep=1.2500cm \numberwithin{equation}{section}
\newcommand{\set}[1]{\left\{#1\right\}}

\newtheorem{Theorem}{Theorem}[section]
\newtheorem{Proposition}[Theorem]{Proposition}
\newtheorem{cor}[Theorem]{Corollary}

\theoremstyle{remark}

\newtheorem{Example}[Theorem]{Example}

\newtheorem{Remark}[Theorem]{Remark}

\begin{document}
\title{On homological notions of Banach algebras related to a character}

   \author{Amir Sahami}
\address{Department of Mathematics, Shahid Rajaee Teacher Training University, Lavizan, Tehran 16788-15811, Iran.} \email{amir.sahami@aut.ac.ir}



\maketitle
%-------------------------------------------------------------

%%%%%%%%%%%%%%%%%%%%%%%%%%%%%%%%%%%%%%%%%%%%%%%%%%%%%%%%%%%%%%%%%%%%%%%%%
\begin{abstract}
In this paper, we countinue our work in \cite{11}.
 We show that $L^{1}(G,w)$ is $\phi_{0}$-biprojective if
and only if $G$ is compact, where $\phi_{0}$ is the augmentation
character. We introduce the notions of character Johnson amenability
and character Johnson contractibility for Banach algebras.
We show that 
$\ell^{1}(S)$ is pseudo-amenable if and only if $\ell^{1}(S)$  is
character Johnson-amenable, provided that $S$ is a uniformly locally finite band semigroup.
 We give some conditions whether $\phi$-biprojectivity
($\phi$-biflatness) of $\ell^{1}(S)$ implies the finiteness
(amenability) of $S$, respectively.
\end{abstract}
\section{Introduction}

 Helemskii studied Banach algebras via the Banach homology theory. In order to his investigation,
 he defined  biflat and biprojective Banach algebras. Indeed, $A$ is called biflat (biprojective), if
there exists a bounded $A$-bimodule morphism $\rho:A\rightarrow
(A\otimes_{p}A)^{**}$ ($\rho:A\rightarrow A\otimes_{p}A$) such that
$\pi^{**}\circ\rho$ is the canonical embedding of $A$ into $A^{**}$
($\rho$ is a right inverse for $\pi$), respectively, see \cite{run}.
He showed that $L^{1}(G)$ is a biflat Banach algebra if and only if
$G$ is amenable and also $L^{1}(G)$ is biprojective if and only if
$G$ is compact, see \cite{Hel}.

Recently, Kanuith  et al. in \cite{Kan} have been used this idea and
defined a new notion of amenability  for Banach algebra depended on
a character  of that Banach algebra. Indeed, for a character
$\phi\in\Delta(A)$, they defined the new notion of left
$\phi$-amenability, that is, $A$ is left $\phi$-amenable Banach
algebra if $\mathcal{H}^{1}(A,X^{*})=\{0\},$ for every Banach
$A$-bimodule $X$, provided that $a\cdot x=\phi(a)x,$ for all $a\in
A$ and $x\in X$. They also showed that the Fourier algebra $A(G)$
is $\phi$-amenable for each $\phi\in\Delta(A)$. Hu {\it et al.} in
\cite{Hu} used the idea of virtual diagonal of Banach algebras and
defined a parallel notion to left $\phi$-amenability and called it
left $\phi$-contractibility. This theory has been under more
investigations, Sangani Monfared in \cite{san} defined the concept of
character amenability which used
 every character of a Banach algebra to studying its properties. He showed that $L^{1}(G)$ is character amenable if and only if $G$ is amenable.  Recently Nasr-Isfahani {\it et al.} has been investigated
 the notions of left $\phi$-amenability and left $\phi$-contractibility in the Banach homology terms, see \cite{nas}.

Motivated by these considerations, in order to find biflatness and biprojectivity related to a character the  author  with A. Pourabbas defined the  notions of
$\phi$-biflatness, $\phi$-biprojectivity and $\phi$-Johnson
amenability for  Banach algebras, see \cite{11}. They showed that
for a locally compact group $G$, $L^{1}(G)$ is $\phi$-biflat if and
only if $G$ is amenable. Also they showed that the Fourier algebra
$A(G)$ is $\phi$-biprojective if and only if $G$ is discrete. For a
discrete group $G$, they showed that $\ell^{1}(G)$ is
$\phi$-biprojective if and only if $G$ is finite.

The content of this paper is as follows, after recalling some
definitions and background notations. We extend \cite[Lemma 4.2]{11} to Beurling algebras. We show that  $L^{1}(G,w)$ is
$\phi_{0}$-biprojective if and only if $G$ is compact, where
$\phi_{0}$ is the augmentation character.) We introduce character
Johnson amenability and character Johnson contractibility for Banach
algebras. We show that 
$\ell^{1}(S)$ is pseudo-amenable if and only if $\ell^{1}(S)$  is
character Johnson-amenable, provided that $S$ is a uniformly locally finite band semigroup.
 We give some conditions whether $\phi$-biprojectivity
($\phi$-biflatness) of $\ell^{1}(S)$ implies the finiteness
(amenability) of $S$, respectively. 
%---------------------------------------------------- ---------------------------------------------------------------------------------------
%%%%%%%%%%%%%%%%%%%%%%%%%%%%%%%%%%%%%%%%%%%%%%%%%%%%%%%%%%%%%%%%%%%%%%%%%%%%%%%%%%%%%%%%%%%%%%%%%%%%%%%%%%%%%%%%%%%%%%%%%%%%%%%%%%%%%%%%%%%
%------------------------------------------------------------------------------------------------------------------------------------------
\section{Preliminaries}
We recall that if $X$ is a Banach $A$-bimodule, then with the
following actions $X^{*}$ is also a Banach $A$-bimodule
$$(a\cdot f)(x)=f(x\cdot a) ,\hspace{.25cm}(f\cdot a)(x)=f(a\cdot x )\qquad(a\in A,x\in X,f\in X^{*}). $$
Let $A$ and  $B$ be Banach algebras. The projective tensor product
of $A$ and $B$ is denoted by $A\otimes_{p} B$ and  with the following
multiplication is a Banach algebra
$$(a_{1}\otimes b_{1})(a_{2}\otimes b_{2})=a_{1}a_{2}\otimes b_{1}b_{2}\qquad (a_{1},a_{2}\in A,\quad b_{1},b_{2}\in B).$$ The Banach algebra
$A\otimes_{p}A$  with the following actions is a Banach $A$-bimodule
$$a\cdot(b\otimes c)=ab\otimes c,~~~(b\otimes c)\cdot a=b\otimes
ca~~~\qquad(a, b, c\in A).$$ Throughout,
the character space of $A$ is denoted by $\Delta(A)$. Let $\phi\in \Delta(A)$.
Then $\phi$ has a unique extension to $A^{**}$ denoted by
$\tilde{\phi}$ and defined by $\tilde{\phi}(F)=F(\phi)$ for every
$F\in A^{**}$. Clearly this extension remains to be a character on
$A^{**}$. We denote $\pi_{A}:A\otimes_{p}A\rightarrow A$ for the
product morphism which specified by $\pi_{A}(a\otimes b)=ab$.

Let $A$ be a Banach algebra and $X$ be a Banach $A$-bimodule.
The $n^{th}$ cohomology group of $A$
with coefficients in $X$ is denoted by $\mathcal{H}^{n}(A,X)$. In fact $A$ is an amenable Banach algebra,
if $\mathcal{H}^{1}(A,X^{*})=\{0\}$ for every Banach $A$-bimodule
$X$.

The Banach algebra $A$ is called $\phi$-biprojective
($\phi$-biflat), if there exists a bounded $A$-bimodule morphism
$\rho:A\rightarrow A\otimes_{p}A\,\,(\rho:A\rightarrow
(A\otimes_{p}A)^{**})$ such that
$$\phi\circ\pi_{A}\circ\rho(a)=\phi(a)\,\,(\tilde{\phi}\circ\pi^{**}_{A}\circ\rho(a)=\phi(a)),$$
 respectively for every $a\in A$.
A Banach algebra $A$ is called $\phi$-Johnson
amenable($\phi$-Johnson contractible) if there exists $m\in
(A\otimes_{p}A)^{**}$ ( $m\in A\otimes_{p}A$) such that $$a\cdot
m=m\cdot a,\quad
\tilde{\phi}\circ\pi^{**}_{A}(m)=1,(\phi\circ\pi_{A}(m)=1)\quad(a\in
A),$$ respectively for every $a\in A$. For more details, we refer
the readers to \cite{11}.

Let $G$ be a locally compact group. A continuous map $w:G\rightarrow
\mathbb{R^{+}}$ is called a weight function, if $w(e)=1$ and for
every $x$ and $y$ in $G$, $w(xy)\leq w(x)w(y)$ and $w(x)\geq1.$ The
Banach algebra of all measurable functions $f$ from $G$ into
$\mathbb{C}$ with $||f||_{w}=\int|f(x)|w(x)dx<\infty$ and the
convolution product is denoted by $L^{1}(G,w)$. The Banach algebra
of all complex-valued, regular and Borel measures $\mu$ on $G$ such
that $||\mu||_{w}=\int_{G}w(x)d|\mu|(x)<\infty$ is denoted by
$M(G,w)$. We write $M(G)$, whenever $w=1.$ The map
$\phi_{0}:L^{1}(G,w)\rightarrow \mathbb{C}$ which specified by
$$\phi_{0}(f)=\int_{G} f(x)dx$$ is called augmentation character,  for more details see
\cite{dale lau}.

We recall that $S$ is an inverse semigroup, if for each $s\in S$
there exists a unique element $s^{*}\in S$ such that $ss^{*}s=s$ and
$s^{*}ss^{*}=s^{*}$
\cite{how}.
The set of
idempotents of a semigroup $S$ is denoted by $E(S)$ . There exists a partial
order on $E(S)$, indeed
$$s\leq t\Longleftrightarrow s=st=ts\qquad( s,t\in E(S)).$$
If $S$ is an inverse semigroup, then there exists a partial order on
$S$ which is coincide with the partial order on $E(S)$. Indeed
$$s\leq t\Longleftrightarrow s=ss^{*}t\qquad (s,t\in
S).$$

 For the partially ordered  set $(S,\leq)$, we denote $(x]=\{y\in
S\,|\,y\leq x\}$. The set  $S$ is called locally finite (uniformly
locally finite) if for every   $x\in S$, we have  $|(x]|<\infty\,\,(\sup\{|(x]||x\in
S\}<\infty)$, respectively.
%------------------------------------------------------------------------------------------------------------------------------------------
%%%%%%%%%%%%%%%%%%%%%%%%%%%%%%%%%%%%%%%%%%%%%%%%%%%%%%%%%%%%%%%%%%%%%%%%%%%%%%%%%%%%%%%%%%%%%%%%%%%%%%%%%%%%%%%%%%%%%%%%%%%%%%%%%%%%%%%%%%%
%------------------------------------------------------------------------------------------------------------------------------------------
\section{$\phi$-biprojectivity of Beurling algebras}
Let $A$ be a Banach algebra and let $L$ be a closed ideal of $A$. We say that $L$ is
left essential as a  Banach $A$-bimodule, if $\overline{AL}=L.$

Let $A$ be a Banach algebra and $\phi\in \Delta(A)$. Suppose that
$L\subseteq\ker\phi$ is a closed ideal of $A$. Clearly $\phi$
induces a character $\overline{\phi}$ on $\frac{A}{L}$, which is
defined by $\overline{\phi}(x+L)=\phi(x)$ for every $x\in A$.
\begin{Proposition}
Let $A$ be a Banach algebra and $\phi\in \Delta(A)$. Suppose that
$A$ is a $\phi$-biprojective Banach algebra and $L\subseteq\ker\phi$
is a closed ideal of $A$ which is   left essential as a Banach
$A$-bimodule. Then $\frac{A}{L}$ is $\overline{\phi}$-biprojective.
\end{Proposition}

\begin{proof}
Since $A$ is a $\phi$-biprojective Banach algebra, there exists a
bounded $A$-bimodule morphism $\rho:A\rightarrow A\otimes_{p}A$ such
that $\phi\circ\pi_{A}\circ\rho(a)=\phi(a)$ for every $a\in A.$ Let
$q:A\rightarrow \frac{A}{L}$ be the quotient map. Define
$\rho_{1}=id\otimes q\circ\rho:A\rightarrow
A\otimes_{p}\frac{A}{L}$. Since $L$ is an essential closed ideal of
$A$, for every $l\in L$, we have $$\rho_{1}(l)=id\otimes
q\circ\rho(l)=id\otimes q\circ\rho(al^{\prime})=id\otimes
q(\rho(a)\cdot l^{\prime})=0,$$ where $l=al^{\prime}$ for some $a\in
A$ and $l^{\prime}\in L$. Hence there exists an induced map (which
still denoted  by $\rho_{1}$) $\rho_{1}:\frac{A}{L}\rightarrow
A\otimes_{p}\frac{A}{L}.$

 Now define $\rho_{2}=q\otimes
id_{\frac{A}{L}}\circ\rho_{1}:\frac{A}{L}\rightarrow
\frac{A}{L}\otimes_{p}\frac{A}{L}.$ We will show that $\rho_{2}$ is
a bounded $\frac{A}{L}$-bimodule morphism and
$\overline{\phi}\circ\pi_{\frac{A}{L}}\circ\rho_{2}(x+L)=\overline{\phi}(x+L).$
Suppose that $x\in A$ and
$\rho(x)=\sum_{i=1}^{\infty}a_{i}^{x}\otimes b_{i}^{x}$ for some
sequences $(a_{i}^{x})_{i}$ and $(b_{i}^{x})_{i}$ in $A$. Then
$\rho_{2}(x+L)=\sum_{i=1}^{\infty}a_{i}^{x}+L\otimes b_{i}^{x}+L$,
so
$\pi_{\frac{A}{L}}\circ\rho_{2}(x+L)=\sum_{i=1}^{\infty}a_{i}^{x}b_{i}^{x}+L$,
therefore
$$\overline{\phi}(\sum_{i=1}^{\infty}a_{i}^{x}b_{i}^{x}+L)=\phi(\sum_{i=1}^{\infty}a_{i}^{x}b_{i}^{x})=\phi\circ\pi_{A}\circ\rho(x)=\phi(x)=\overline{\phi}(x+L).$$
Now suppose that $a+L$ is an arbitrary element of $\frac{A}{L}$. Then
$a+L\cdot\rho_{2}(x+L)=\sum_{i=1}^{\infty}aa_{i}^{x}+L\otimes
b_{i}^{x}+L$. Since $\rho$ is a left $A$-module
morphism, $\rho_{1}$ is a left $A$-module morphism. Hence
\begin{equation*}
\begin{split}
\rho_{2}(a x+L)=q\otimes id_{\frac{A}{L}}\circ \rho_{1}(a x+L)
&=q\otimes id_{\frac{A}{L}}(a\cdot \rho_{1}(x+L))\\
&=q\otimes id_{\frac{A}{L}}(\sum_{i=1}^{\infty}aa^{x}_{i}\otimes
b^{x}_{i}+L)\\
&=\sum_{i=1}^{\infty}aa^{x}_{i}+L\otimes b^{x}_{i}+L\\
&=a+L\cdot \sum_{i=1}^{\infty}a^{x}_{i}+L\otimes
b^{x}_{i}+L\\
&=a+L\cdot\rho_{2}(x+L).
\end{split}
\end{equation*}
Similarly one can show that $\rho_{2}$ is a right
$\frac{A}{L}$-module morphism and the proof is complete.
\end{proof}

We recall that $m\in A\otimes_{p}A $ is a $\phi$-Johnson contraction
for $A$, if $a\cdot m=m\cdot a$ and $\phi\circ\pi_{A}(m)=1$, where
$a\in A$, for more details the reader referred to \cite {11}.

Let $A$ be a Banach algebra and $\phi\in\Delta(A)$. $A$ is left
$\phi$-contractible if and only if there exists an element $m$ in
$A$ such that $am=\phi(a)m$ and $\phi(m)=1$,  see   \cite{Hu} and
\cite{nas}.
Note that the left  $\phi$-contractibility of a Banach algebra  $A$  is equivalent to property that;  the Banach algebra $\mathbb{C}$ is a  projective left Banach $A$-module
with the following left action,  $a\cdot z=\phi(a)z$ for every $a\in A$ and $z\in \mathbb{C}$
\cite[Theorem 4.3]{nas}.

Compare the following Theorem with \cite[Theorem 5.13]{Hel}.
\begin{Theorem}\label{main}
Let $G$ be a locally compact group, let $\omega$ be a weight on $G$ and let $\phi_{0}$ be the augmentation character on $L^{1}(G,w)$. Then the following
are equivalent
\begin{enumerate}
\item [(i)] $L^{1}(G,w)$ is $\phi_{0}$-biprojective;
\item [(ii)] $L^{1}(G,w)$ is left $\phi_{0}$-contractible;
\item [(iii)] $G$ is compact.
\end{enumerate}
\end{Theorem}
\begin{proof}
(i)$\Rightarrow$(ii) Set $A=L^{1}(G,w)$ and $L=\ker\phi_{0}$. Let
$A$ be $\phi_{0}$-biprojective. Since $A$ has a bounded approximate
identity, $L$ becomes a left essential  Banach $A$-bimodule. Thus by
the proof of previous Proposition there exists a bounded left
$A$-module morphism $$\rho_{1}:\frac{A}{L}\rightarrow
A\otimes_{p}\frac{A}{L}.$$ Since $\frac{A}{L}\cong\mathbb{C}$, hence
we have $\rho_{1}:\mathbb{C}\rightarrow A\otimes_{p}\mathbb{C}\cong
A$ such that
$\overline{\phi_{0}}\circ\pi_{A,\frac{A}{L}}\circ\rho_{1}(c)=
\overline{\phi_{0}}(c)$, where $c\in \mathbb{C}$. Set $m=\rho(1)\in
A$. Then
$\phi_{0}(m)=\phi_{0}(\rho(1))=\overline{\phi_{0}}\circ\pi_{A,\frac{A}{L}}\circ\rho(1)=1$
and $a\cdot\rho(1)=\rho(a\cdot1)=\phi_{0}(a)\rho(1)$, where $a\in
A$. Hence $A$ is left $\phi_{0}$-contractible.

~~(ii)$\Rightarrow$(iii) Suppose that $A$ is a left
$\phi_{0}$-contractible Banach algebra. Then there exists an element
$m\in A$  such that $am=\phi_{0}(a)m$ and $\phi_{0}(m)=1$, where
$a\in A$. Let $g\in G$ be an arbitrary element and $f\in A \setminus L$. Hence
$$\phi_{0}(f)\delta_{g}\ast m=\delta_{g}\ast(f\ast
m)=(\delta_{g}\ast f)\ast m=\phi_{0}(\delta_{g}\ast
f)m=\phi_{0}(f)m.$$ Hence $m$ is constant and belongs to $A$, which implies that $\int_{G} w(x)dx<\infty$. Therefore
$$|G|=\int_{G} w(e) dx<\infty ,$$
so $G$ is a compact group.

~~(iii)$\Rightarrow$(i) Let $G$ be a compact group  and consider a normalized left Haar measure.
Then $m=1\otimes 1$
in $A\otimes_{p}A$ satisfies $a\cdot m=m\cdot a=\phi_{0}(a)m$ and
$\phi_{0}\circ\pi_{A}(m)=1$,  where $a\in A$. Thus $A$ is  $\phi_{0}$-Johnson contractible.
Hence \cite[Lemma 3.2]{11} gives $\phi_0$-biprojectivity of $A$.
\end{proof}

It is easy to see that every biprojective Banach algebra $A$ is
$\phi$-biprojective for every $\phi\in\Delta(A)$, but the converse
is not always true. On the other hand \cite[Theorem 5.2.30]{run}
asserts that, if $A$ is biprojective, then for every Banach
$A$-bimodule $X,$ $\mathcal{H}^{n}(A, X)=0$, where $n\geq 3$. This
question maybe asked "what will happen, if $A$ is
$\phi$-biprojective?" at the following corollary we answer this
question  for the group algebras.
\begin{cor}
Let $G$ be a locally compact group.
\begin{enumerate}
\item [(i)] If $L^{1}(G)$ is  $\phi_{0}$-biprojective, then  for every Banach $L^{1}(G)$-bimodule $X,$ $\mathcal{H}^{n}(L^{1}(G), X)=0$, where $n\geq 3$.
\item [(ii)] $L^{1}(G)$ is $\phi_{0}$-biprojective if and only if $\mathcal{H}^{1}(L^{1}(G),X)=0$, for every Banach $L^{1}(G)$-bimodule  $X$ with $x\cdot a=\phi_{0}(a)x$ such that $a\in L^{1}(G)$ and $x\in X.$
\end{enumerate}
\end{cor}
\begin{proof} (i)
Let   $L^{1}(G)$  be  $\phi_{0}$-biprojective. Then  by Theorem \ref{main}
$G$ is compact and \cite{run} shows that $L^1(G)$ is biprojective for every compact group $G$.
Now using  \cite[Theorem 5.2.30]{run} one can get the results.

 (ii) holds by Theorem \ref{main}.
\end{proof}
For a Banach algebra $A$, $db A$ denoted for the minimum values of $n\in \mathbb{Z}^{+}$ such that $A^{\sharp}$ has a projective resolution of length $n$, see \cite[page 294]{dale auto}. Helemskii showed that for a biprojective Banach algebra $A$, $dbA\leq 2$, see \cite[Theorem 2.8.56]{dale auto}. Also it is well-known that $L^{1}(G)$ is biprojective if and only if $G$ is compact. Combine these facts and the previous Corollary  one can see that  if $L^{1}(G)$ is $\phi_{0}$-biprojective, then  $db L^{1}(G)\leq 2$.

%------------------------------------------------------------------------------------------------------------------------------------------
%%%%%%%%%%%%%%%%%%%%%%%%%%%%%%%%%%%%%%%%%%%%%%%%%%%%%%%%%%%%%%%%%%%%%%%%%%%%%%%%%%%%%%%%%%%%%%%%%%%%%%%%%%%%%%%%%%%%%%%%%%%%%%%%%%%%%%%%%%%
%------------------------------------------------------------------------------------------------------------------------------------------
\section{$\phi$-homological properties of semigroup algebras}
We remind that $S$ is a left (right) amenable semigroup if there exists an
element $m\in \ell^{1}(S)^{**}$ such that
$$s\cdot m=m\quad(m\cdot s=m),\qquad ||m||=m(\phi)=1\hspace{1cm}( s\in S),$$
where $\phi$ is the augmentation character of  $\ell^{1}(S)$, respectively. The semigroup $S$ is called amenable,  if it is both left and right amenable.

We recall that $S$ is a band semigroup, if $S=E(S)$. A band
semigroup $S$ is called rectangular band if $xyx=x$, for every
$x,y\in S.$ In this case there exists an equivalence relation on
$S$, in fact
$$a\mathcal{R}b\Longleftrightarrow S^{1}aS^{1}=S^{1}bS^{1},\hspace{1cm}(a,b\in S),$$
where $S^{1}=S\cup\set{1}$ \cite{how}.
Let $A$ be a Banach algebra and $\Lambda$ be a semilattice. Suppose that $\{A_{\lambda}:\lambda\in\Lambda\}$ is a collection of closed subalgebra of $A$. If $A$ is a $\ell^{1}$-direct sum of $A_{\lambda}$ as a Banach space and $A_{\lambda_{1}}A_{\lambda_{2}}\subseteq A_{\lambda_{1}\lambda_{2}}$, then $A$ is called $\ell^{1}$-graded of $A_{\lambda}$'s and denoted by $\oplus^{\ell^{1}}_{\lambda}A_{\lambda}.$

We say that $A$ is character-Johnson amenable (character-Johnson contracatible), if for every $\phi\in\Delta(A)$, $A$ is $\phi$-Johnson amenable ($\phi$-Johnson contractible), respectively.
\begin{Theorem}\label{band}
Suppose that $S$ is a band semigroup. Let $\ell^{1}(S)$ be character
Johnson-amenable. Then $S$ is a semilattice, so is amenable.
\end{Theorem}
\begin{proof}
Let $S$ be  a band semigroup. Then by \cite[Theorem 4.4.1]{how}
 $S=\cup_{\lambda\in \Lambda}S_{\lambda}$, where $S_{\lambda}$ is a  rectangular band semigroup for every  $\Lambda\in \Lambda$. Since
$S_{\lambda_{1}}S_{\lambda_{2}}\subseteq
S_{\lambda_{1}\lambda_{2}}$, we have
$\ell^{1}(S)=\oplus^{\ell^{1}}_{\lambda}\ell^{1}(S_{\lambda})$, here the index set
$\Lambda$ is a semilattice.

Set  $I=\oplus^{\ell^{1}}_{\lambda\leq
\lambda_{0}}\ell^{1}(S_{\lambda})$, where $\lambda_{0}\in\Lambda$ is fixed. One can easily see that $I$ is a
closed ideal of $\ell^{1}(S)$. Since $\ell^{1}(S_{\lambda_{0}})$ is
a homomorphic image of $I$.
For every $\phi\in\Delta(\ell^{1}(S_{\lambda_{0}}))$  we take
 $\phi\circ\eta$ as a character on $I$, which we denote it by
$\phi_{I}$, where  $\eta:I\rightarrow
\ell^{1}(S_{\lambda_{0}})$ is a homomorphism with a dense range.
 It is easy to see that $\phi_{I}$ can  be extend to
$\ell^{1}(S)$ which is denoted by $\phi_{S}$

Moreover,  there exists an isomorphism between $S_{\lambda_{0}}$ and
$L\times R$,  where $L$ and $R$ are denoted for a left-zero semigroup and a
right-zero semigroup, respectively  \cite[Theorem 1.1.3]{how}. Also  we have
$$\ell^{1}(S_{\lambda_{0}})\cong \ell^{1}(L\times R)\cong\ell^{1}(L)\otimes_{p}\ell^{1}(R).$$
Take $\phi=\phi_{0}\otimes\sigma_{0 }\in \Delta(\ell^{1}(S_{\lambda_{0}}))$, where $\phi_{0}$ and $\sigma_{0}$
are the augmentation characters on $\ell^{1}(L)$ and $\ell^{1}(R)$,
respectively. Consider  $\phi_{I}$ and $\phi_{S}$ corresponding to $\phi$ as before. Since $\ell^{1}(S)$ is character Johnson-amenable, by
\cite[Proposition 2.2]{11} $\ell^{1}(S)$ is left $\phi_{S}$-amenable
and right $\phi_{S}$-amenable. Since
$\phi_{S}|_{\ell^{1}(S_{\lambda_{0}})}\neq 0$, we have $\phi_{I}\neq
0$, so by \cite[Lemma 3.1]{Kan} $I$ is  left $\phi_{I}$-amenable and
right $\phi_{I}$-amenable.  But, since
$\ell^{1}(S_{\lambda_{0}})$ is a homomorphic image of $I$,  by
\cite[Proposition  3.5]{Kan} $\ell^{1}(S_{\lambda_{0}})$ is left
$\phi$-amenable and right $\phi$-amenable. Hence by \cite[Theorem
3.3]{Kan} $\ell^{1}(L)$ is left $\phi_{0}$-amenable and
$\ell^{1}(R)$ is right $\sigma_{0}$-amenable.  So  \cite[Theorem
 1.4]{Kan}  shows that there exists a net  $(m_{\alpha})_{\alpha}$ in
$\ell^{1}(L)$ such that
\begin{equation}\label{equ}
\begin{split}
am_{\alpha}-\phi_{0}(a)m_{\alpha}\xrightarrow{||\cdot||}0,\quad
\phi(m_{\alpha})=1.
\end{split}
\end{equation}
Replace $a_{1}=\delta_{s_{1}}$ and
$a_{2}=\delta_{s_{2}}$ in (\ref{equ}) instead of $a$, respectively for every $s_{1}, s_{2}\in L$.
 One can see that
$m_{\alpha}\rightarrow \delta_{s_{1}}$ and $m_{\alpha}\rightarrow
\delta_{s_{2}}$, which implies that $L$ and similarly $R$ are singleton, then $S_{\lambda_{0}}$ is singleton
and therefore  with the same argument we can show that $S_{\lambda}$ is singleton for every $\lambda\in \Lambda$.
Hence    $S=\cup_{\lambda\in \Lambda}S_{\lambda}$
is isomorphic to
$\Lambda$. Since every semilattice is commutative, $S$ is amenable and  the proof is complete.
\end{proof}
We recall that $A$ is a  pseudo-amenable Banach algebra, if there
exists a (not necessarily bounded) net $(m_{\alpha})_{\alpha}$ in
$A\otimes_{p}A$ such that $a\cdot m_{\alpha}-m_{\alpha}\cdot
a\xrightarrow{||\cdot||}0$ and
$\pi_{A}(m_{\alpha})a\xrightarrow{||\cdot||}a$, for every $a\in A$
\cite{ghah pse}.

Using \cite[Corollary 3.5]{rost1} and previous Theorem, we get the
following corollary.
\begin{cor}
Let $S$ be a uniformly locally finite band semigroup. Then
$\ell^{1}(S)$ is pseudo-amenable if and only if $\ell^{1}(S)$  is
character Johnson-amenable.
\end{cor}

Note that in the general case the pseudo-amenability is not  equivalent with the
character Johnson-amenability. To see this we give the following
example.

\begin{Example}
 Suppose that $G$ is a compact infinite group. Then by
\cite[Proposition 4.2]{ghah pse} $L^{1}(G)^{**}$ is not
pseudo-amenable. The set of all continuous
character  $\rho:G\rightarrow \mathbb{T}$ is denoted by $\widehat G$. It is well-known that every character
$\phi\in\Delta(L^{1}(G))$ is of the form
$$\phi_{\rho}(f)= \int_{G}\overline{\rho(x)}f(x)dx,$$ where $dx$ is
a left Haar measure on $G$, for more details, see \cite[Theorem
23.7]{20}. It is also well-known that $\phi_{\rho}$ has a unique
extension to ${L^{1}(G)}^{**}$, which denoted by
$\tilde{\phi}_{\rho}$. Hence $\Delta({L^{1}(G)}^{**})$ consists of
all $\tilde{\phi}_{\rho}$, for every $\rho\in\widehat G$. Since $G$ is compact,
$\widehat G\subset L^{\infty}(G)\subseteq L^{1}(G)$.  Define
$m_{\rho}=\rho\otimes\rho\in L^{1}(G)\otimes_{p}L^{1}(G)$. Since two
maps $a\mapsto a\rho$ and $a\mapsto \rho a$ are $w^{*}$-continuous
on ${L^{1}(G)}^{**}$ for every $a\in {L^{1}(G)}^{**}$, one can
easily see that $a\cdot m_{\rho}=m_{\rho}\cdot a$ and
$\tilde{\phi}_{\rho}\circ\pi_{{L^{1}(G)}^{**}}^{**}(m_{\rho})=1.$
Hence ${L^{1}(G)}^{**}$ is character Johnson-amenable.
\end{Example}
It is well-known that for an inverse semigroup $S$ there exists an
equivalence relation $\mathcal{R}$ on $S$, that is, for every $x,
y\in S$, $x\mathcal{R} y$ if and only if there exits $e\in E(S)$
such that $es=et$. Consider $G_{S}=\frac{S}{\mathcal{R}}$  \cite{mun}.
\begin{Proposition}
Let $S$ be an inverse semigroup. If $\ell^{1}(S)$ is character
Johnson-amenable, then $G_{S}$ is an amenable group.
\end{Proposition}
\begin{proof}
Since $G_{S}$ is a quotient of $S$, then $\ell^{1}(G_{S})$ is a
homomorphic image of $\ell^{1}(S)$. Suppose that
$\phi\in\Delta(\ell^{1}(G_{S}))$  and $p:\ell^{1}(S)\rightarrow
\ell^{1}(G_{S})$ is a dense range homomorphism. Since $\ell^{1}(S)$
is character Johnson amenable, $\ell^{1}(S)$ is $\phi\circ
p$-Johnson amenable. Now by  \cite[Proposition 2.2]{11}, $\ell^{1}(S)$ is left $\phi\circ
p$-amenable. Hence \cite[Proposition 3.5]{Kan} shows that $\ell^{1}(G)$ is left $\phi$-amenable.
Now by applying  \cite[Corollary 3.4]{alagh} $G_{S}$
must be amenable.
\end{proof}
Let $G$ be a group and $I$ be a non-empty set. Set
$\mathcal{M}^{0}(G,I)=\{(g)_{ij}|g\in G, i,j\in I\}\cup\{0\}$, where $(g)_{ij}$ is denoted for $I\times I$ matrix with
entry $g$ in $(i,j)^{th}$-position and zero elsewhere. With
the following multiplication $\mathcal{M}^{0}(G,I)$ is a semigroup
$$(g)_{ij}(h)_{kl}=\left \{\begin{array}{cc} (gh)_{il}&\mbox{if}\quad j=k\\
0&\mbox{if}\quad j\neq k\\
\end{array}
\right.$$
This semigroup is called  Brandt semigroup over $G$ with index set $I$. It is well-known that for $S=\mathcal{M}^{0}(G,I)$, $G_{S}=G$.
\begin{cor}
Let $G$ be a group and $I$ be a non-empty set and also let $S=\mathcal{M}^{0}(G,I)$. If $\ell^{1}(S)$ is character Johnson amenable, then $\ell^{1}(S)$ is pseudo-amenable.
\end{cor}
\begin{proof}
Let $\ell^{1}(S)$ be character Johnson amenable. By previous Proposition $G_{S}=G$ must be amenable. Now apply \cite[Corollary 3.8]{rost} to show that $\ell^{1}(S)$ is pseudo-amenable.
\end{proof}
\begin{Proposition}
Let $S$ be an inverse semigroup. If $\ell^{1}(S)$ is character
Johnson-contractible, then $G_{S}$ is finite.
\end{Proposition}
\begin{proof}
Use the  same argument  as in the proof of pervious Proposition and the
fact that $\ell^{1}(G_{S}) $ is left $\phi$-contractible if and only
if $G_{S}$ is finite, see \cite[Theorem 3.3]{alagh}.
\end{proof}
The results of
the previous  two  Propositions holds when we replace the hypothesis
``$A$ is left $\phi$-amenable ($\phi$-contractible)" instead  of
 ``$A$ is character Johnson amenable (character Johnson contractible)"
  respectively for every $\phi\in\Delta(A)$.
\begin{Proposition}
Let $S$ be a semigroup such that its center  $Z(S)$ is non-empty. If $\ell^{1}(S)$ is $\phi$-biflat, then $S$ is
amenable, where $\phi$ is the augmentation character on $\ell^{1}(S)$.
\end{Proposition}
\begin{proof}
Suppose that  $\ell^{1}(S)$
is $\phi$-biflat, where $\phi$ is the augmentation character on
$\ell^{1}(S)$.
 Let $\rho:\ell^{1}(S)\rightarrow
(\ell^{1}(S)\otimes_{p}\ell^{1}(S))^{**}$ be a bounded
$\ell^{1}(S)$-bimodule morphism such that
$\tilde{\phi}\circ\pi^{**}_{\ell^{1}(S)}\circ\rho(a)=\phi(a),$ for
every $a\in \ell^{1}(S)$. Set $m_{0}=\rho(\delta_{s_{0}})$, where $s_0\in Z(S)$,
 it is easy to see that $\delta_{s}\cdot m_{0}=m_{0}\cdot\delta_{s}$ and
$\tilde{\phi}\circ\pi^{**}_{\ell^{1}(S)}(m_{0})=1.$ Then
$\ell^{1}(S)$ is $\phi$-Johnson amenable. Applying the same
arguments as in the proof of \cite[Proposition 2.2]{11}, one can
show that $\delta_{s}\cdot m_{0}=m_{0}\cdot\delta_{s}=m_{0}$ and
$\tilde{\phi}\circ\pi^{**}_{\ell^{1}(S)}(m_{0})=1.$ Suppose that
$m=\pi^{**}_{\ell^{1}(S)}(m_{0})\in \ell^{1}(S)^{**}$. Hence we have
$\delta_{s} m=m \delta_{s}=m$ and $\tilde{\phi}(m)=1.$   Hence $S$
is an amenable semigroup, see \cite[Theorem 1.1]{Kan}.
\end{proof}
\begin{Proposition}\label{finite}
Let $S$ be a semigroup such that $Z(S)$ is non-empty. If
$\ell^{1}(S)$ is $\phi$-biprojective and $S$ has left or right
unit, then $S$ is finite,  where $\phi$ is the augmentation character on
$\ell^{1}(S)$.
\end{Proposition}
\begin{proof}
Suppose that $\ell^{1}(S)$ is  $\phi$-biprojective, where $\phi$ is the
augmentation character on $\ell^{1}(S)$. Then there exists a bounded $\ell^{1}(S)$-bimodule
morphism $\rho:\ell^{1}(S)\rightarrow
\ell^{1}(S)\otimes_{p}\ell^{1}(S)$ such that
$\phi\circ\pi_{\ell^{1}(S)}\circ\rho(a)=\phi(a),$ for
every $a\in \ell^{1}(S)$.

Define $m=\pi_{\ell^{1}(S)}\circ\rho(\delta_{s_{0}})$, where $s_{0}\in Z(S)$. Then we have
$\delta_{s} m=m\delta_{s}=m$ and $\phi(m)=1.$  Now if $e_{r}$ is a right unit for $S$, then
for every $s\in S$ we have
$$m(s)=m(se_{r})=\delta_{s}m(e_{r})=m(e_{r}),$$ that is,
 $m\in \ell^{1}(S)$ is a constant function on $S$ , so $S$ must be  finite.
\end{proof}
\begin{Remark}
There exists a biprojective semigroup algebra which is not character
Johnson amenable. To see this let $S$ be an infinite left zero
semigroup, that is, $st=s$ for every $s,t \in S$. It is easy to see
that $$fg=\phi_{S}(g)f,\quad f,g\in \ell^{1}(S),$$ where $\phi_{S}$
is the augmentation character on $\ell^{1}(S)$. Define
$\rho:\ell^{1}(S)\rightarrow \ell^{1}(S)\otimes_{p}\ell^{1}(S)$ by
$\rho(f)=f\otimes f_{0}$. It is easy to see that $\rho$ is a bounded
$A$-bimodule morphism which $\pi_{\ell^{1}(S)}\circ\rho(f)=f$ for
every $f\in \ell^{1}(S).$ It follows that $\ell^{1}(S)$ is
biprojective. Now using the same method as in the proof of
\ref{band} one can see that $\ell^{1}(S)$ is not character Johnson
amenable. Note that in the  previous Proposition
  the hypothesis ``$Z(S)\neq \emptyset$" is necessary. It is easy to see that for a left zero semigroup $S,$
  $Z(S)=\emptyset$. Also one can show that
for the augmentation character $\phi$, $\ell^{1}(S)$ is
$\phi$-biprojective,
 but $S$ is not finite.

Also note that  the hypothesis ``existence of left or right unit"
is necessary. To see this let $S=\mathbb{N}$ with the
product $m\cdot n=\min\{m,n\}\,\,(m,n\in S)$ which is an infinite semigroup with no unit such that $Z(S)=S$  \cite[Example 5.2]{11} and $\ell^{1}(S)$ is
$\phi$-biprojective, where $\phi$ is the augmentation
character.
\end{Remark}

%------------------------------------------------------------------------------------------------------------------------------------------
%%%%%%%%%%%%%%%%%%%%%%%%%%%%%%%%%%%%%%%%%%%%%%%%%%%%%%%%%%%%%%%%%%%%%%%%%%%%%%%%%%%%%%%%%%%%%%%%%%%%%%%%%%%%%%%%%%%%%%%%%%%%%%%%%%%%%%%%%%%
%------------------------------------------------------------------------------------------------------------------------------------------
%\begin{small}
%\begin{thebibliography}{99}

%\end{small}

\begin{thebibliography}{HD}
\bibitem{alagh} M. Alaghmandan,  R. Nasr Isfahani,  and  M. Nemati, {\it Character amenability and contractibility of abstract Segal algebras}, Bull. Aust. Math. Soc, {\bf 82} (2010) 274-281.

\bibitem{dale auto}  H. G. Dales,  {\it Banach algebras and automatic continuity}, Clarendon Press, Oxford, 2000.
\bibitem{dale lau}  H. G. Dales  and  A. T. Lau, {\it The second duals of Beurling algebras},
Mem. Amer. Math. {\bf 177} (2005).

\bibitem{rost1}  M. Essmaili, M. Rostami, and A. R. Medghalchi,  {\it Pseudo-contractibility and pseudo-amenability of semigroup
algebras}, Arch. Math {\bf 97} (2011), 167-177.

\bibitem{rost} M. Essmaili, M. Rostami  and A. Pourabbas,  {\it Pseudo-amenability of certain semigroup
algebras}, Semigroup Forum (2011), 478-484.
\bibitem{for}B. E.  Forrest  and V. Runde,  {\it Amenability and weak amenability of the
Fourier algebra}, Math. Z. {\bf 250} (2005), 731-744.
\bibitem{ghah pse} F. Ghahramani  and Y. Zhang,  {\it Pseudo-amenable and pseudo-contractible Banach algebras}, Math. Proc. Camb. Phil.
Soc. {\bf 142} (2007) 111-123.

\bibitem{Hel} A. Ya. Helemskii,  {\it The homology of Banach and topological
algebras}, Kluwer, Academic Press, Dordrecht, 1989.
\bibitem{20} E. Hewitt  and K. Ross,  {\it Abstract Harmonic Analysis I, Die Grundlehren der Mathematischen
Wissenschaften}, {\bf 115}, Springer-Verlag, Berlin, (1963).
\bibitem{how}  J. Howie,  {\it Fundamental of Semigroup Theory}. London Math. Soc
Monographs, vol. {\bf 12}. Clarendon Press, Oxford (1995).
\bibitem{Hu} Z. Hu  M. Sangani Monfared and T. Traynor, {\it On character amenable Banach algebras}, Studia Math. {\bf 193} (2009) 53-78.
\bibitem{Kan} E.  Kaniuth  A. T. Lau  and J. Pym,  {\it On $\phi$-amenablity of Banach algebras}, Math. Proc. Camb. Soc. {\bf 44} (2008) 85-96.
\bibitem{mun}  W. D. Munn, {\it A class of irrecucible matrix representaions of an arbitrary
inverse semigroup}, Proc. Glasgow Math. Assoc. {\bf 5} (1961),
41-48.
\bibitem{nas} R. Nasr Isfahani  and S. Soltani Renani, {\it Character contractibility of Banach algebras and homological properties of Banach modules}, Studia
Math. {\bf 202} (2011) 205-225.
\bibitem{run} V.  Runde,  {\it Lectures on amenability}, Springer, New York, 2002.
\bibitem{11} A. Sahami  and A. Pourabbas,  {\it On $\phi$-biflat and $\phi$-biprojective  Banach
algebras}, Bull. Belg. Math. Soc. Simon Stevin, {\bf 20}(2013)
789-801.


\bibitem{sam}E.  Samei, N. Spronk  and R. Stokke,  {\it Biflatness and pseudo-amenability of Segal algebras}, Canad.
J. Math. {\bf 62} (2010), 845-869.

\bibitem{san} M. Sangani Monfared,  {\it Character amenability of Banach
algebras}, Math. Proc. Camb. Phil. Soc. {\bf 144} (2008) 697-706.

\end{thebibliography}
\end{document}